\documentclass[12pt]{article}
\usepackage{amssymb,amsmath, amsfonts, amsthm}
\usepackage{epsf,tikz}
\usepackage{caption, subcaption}
\usepackage{graphicx}
\newtheorem{theorem}{Theorem}

\newtheorem{proposition}[theorem]{Proposition}
\newtheorem{lemma}[theorem]{Lemma}
\newtheorem{cor}[theorem]{Corollary}
\newcommand{\gid}{\gamma^{\rm ID}}

\begin{document}
\title{Identifying codes of the direct product of two cliques}

\author{Douglas F. Rall\thanks{The second author is Herman N. Hipp Professor of Mathematics
at Furman University.
This work was partially supported by a grant from the Simons Foundation
(\#209654 to Douglas Rall).}
\\
Furman University\\
Greenville, SC, USA\\
doug.rall@furman.edu
\and
Kirsti Wash \\
Clemson University\\
Clemson, SC, USA\\
kirstiw@g.clemson.edu
}

\date{\today}
\maketitle

\begin{abstract}
An identifying code in a graph is a dominating set that also has the property that the closed
neighborhood of each vertex in the graph has a distinct intersection with the set.  The
minimum cardinality of an identifying code in a graph $G$ is denoted $\gid(G)$. It was recently
shown by Gravier, Moncel and Semri that $\gid(K_n \Box K_n) = \lfloor{\frac{3n}{2}\rfloor}$.
Letting $n, m \ge 2$ be any integers, we consider identifying codes of the direct product
$K_n \times K_m$. In particular, we answer a question of Klav\v{z}ar and show the exact value
of $\gid(K_n \times K_m)$.
\end{abstract}



\section{Introduction} \label{intro}
An identifying code in a graph is a dominating set that also has the property
that the closed neighborhood of each vertex has a distinct intersection with the
set.  Because of this characteristic of the dominating set every vertex can be
uniquely located by using this intersection with the identifying code.  The first
to study identifying codes were Karpovsky, Chakrabarty and Levitin \cite{IDCintro}
who used them to analyze fault-detection problems in multiprocessor systems.
An excellent, detailed list of references on identifying codes  can be found on
Antoine Lobstein's webpage \cite{Lobweb}.  The
usual invariant of interest is the minimum cardinality of an identifying code in
a given graph.  In this regard various families  of graphs have been studied,
including trees~\cite{Lobweb}, paths~\cite{bchl2004}, cycles~\cite{bchl2004,gms2006,xth2008},
and infinite grids~\cite{bl2005,chmglpz1999,idcodegrids}.

In terms of graph products, a  few of the more recent results have been in the study of
 hypercubes~\cite{bhl2000, hl2002,idcodehyper,kcla1999,m2006},
 the Cartesian product of two same size cliques~\cite{S.-Gravier:2008uq}, and the lexicographic
 product of two graphs~\cite{Feng:2012fk}. A natural problem (posed by
 Klav\v{z}ar~\cite{k2011} at the
 Bordeaux Workshop on Identifying Codes in 2011) is to determine the order of a minimum
 identifying code in the direct product of two complete graphs.  In this paper we completely solve
 this problem.

The remainder of the paper is organized as follows.  We first give some useful
definitions and terminology.  In Section~\ref{sec:mainresults} we state the main results
which give the cardinality of a minimum identifying code for the direct product of
any two nontrivial cliques.  Section~\ref{sec:properties} is devoted to deriving
some important properties that will be useful in showing that a set of vertices is
an ID code in a direct product of 2 cliques.  Then the proofs of the main results
are given in Section~\ref{sec:proofs}.

\subsection{Definitions and Notation} \label{sec:defns}

Given a simple undirected graph $G$ and a vertex $x$ of $G$, we let $N(x)$ denote the
\textit{open neighborhood} of $x$, that is, the set of vertices adjacent to $x$.
The \textit{closed neighborhood} of $x$ is  $N[x]=N(x) \cup \{x\}$.  A subset $D \subseteq V(G)$
is a \textit{dominating set} of $G$ if $D$ has a nonempty intersection with the closed
neighborhood of every vertex of $G$. A subset $S \subseteq V(G)$  \textit{separates}
two distinct vertices $x$ and $y$ if $N[x] \cap S \ne N[y] \cap S$. When $S=\{u\}$ we
say that $u$ separates $x$ and $y$.  An \textit{identifying code}
(\textit{ID code} for short) of $G$ is a subset $C$ of vertices that is a dominating set of
$G$ with the additional property that $C$ separates every pair of distinct vertices of $G$.
The minimum cardinality of an ID code of $G$ is  denoted $\gid(G)$.  If $C$ is an ID code
of $G$, then any vertex in $C$ is called a \textit{codeword}.  Note that any
graph having two vertices with the same closed neighborhood (so-called \textit{twins})
does not have an ID code.

Given two graphs $G_1 = (V_1, E_1)$ and $G_2 = (V_2, E_2)$, the \textit{direct product}
of $G_1$ and $G_2$, denoted $G_1 \times G_2$, is the graph whose vertex set
is the Cartesian product, $V_1 \times V_2$, and whose edge set is
$E(G_1 \times G_2) = \{(u_1, u_2)(v_1, v_2)\,|\,u_1v_1 \in E_1 \text{ and } u_2v_2\in E_2\}$.
Direct products have been studied for some time, and extensive information on their
structural properties can be found in \cite{HPG}.

For a positive integer $n$ we write $[n]$ to denote the set $\{1,2,\ldots,n\}$, and $[n]$
will be the vertex set of the complete graph $K_n$. In the direct product $K_n \times K_m$
we refer to a {\it column} as the set of all vertices having the same first coordinate.
A {\it row} is the set of all vertices with the same second coordinate.  In particular, for
$i \in [n]$, the $i^{\rm th}$ column is  $C_i =\{(i,j)\,|\, j \in [m]\}$.  Similarly, for
$j \in [m]$ the $j^{\rm th}$ row is the set $R_j = \{(i,j) \,|\, i \in [n]\}$.  In any
figures rows will be horizontal and columns vertical.  For ease of reference in this paper
we refer to $K_n$ as the \textit{first factor} of $K_n \times K_m$  and $K_m$ as the
\textit{second factor}.  The 2 product graphs $K_n \times K_m$ and $K_m \times K_n$
are clearly isomorphic under a natural map.  Throughout the remainder of this work
we always have the smaller factor first.

Let $G=K_n \times K_m$ and suppose that $C \subseteq V(G)$.  The \textit{column span of $C$}
is the set of all columns of $G$ that have a nonempty intersection with $C$.  The number of
columns in the column span of $C$ is denoted by $cs(C)$.  Similarly, the set of all rows of $G$
that contain at least one member of $C$ is the \textit{row span of $C$}; its size is denoted
$rs(C)$.  For a vertex $v=(i,j)$ of $G$ we say that $v$ is \textit{column-isolated
in $C$} if $C \cap C_i = \{v\}$.  Similarly, if $C \cap R_j = \{v\}$ then
we say that $v$ is \textit{row-isolated in $C$}. If $v$ is both column-isolated and
row-isolated in $C$, we simply say $v$ is \textit{isolated in $C$}.  When there is no chance
of confusion and the set $C$ is clear from the context we shorten these to column-isolated,
row-isolated and isolated, respectively.  

\section{Main Results} \label{sec:mainresults}

\vskip5 mm
In this paper we determine the minimum cardinality of an identifying code
for the direct product of any two nontrivial complete graphs.  We prove
the following results.  Note that $K_2 \times K_2$ has vertices with
identical closed neighborhoods and so has no ID code.

\begin{theorem} \label{thm:nequals2}
For any positive integer $m \ge 5$, $\gid(K_2 \times K_m)=m-1$. In addition,
if $3 \le m \le 4$, $\gid(K_2\times K_m)=m$.
\end{theorem}

For $3 \le n \le 5$ and $n \le m \le 2n-1$ the values of
$\gid(K_n \times K_m)$ were computed by computer program and are given in the following table.

\begin{table}[ht!]
\begin{center}
$
\begin{array}{c|*{11}{p{5mm}}l}
n \backslash ^{\textstyle m} & 3 & 4 & 5 & 6 & 7 & 8 & 9 \\ \hline
3 & 4 & 4 & 5 \\
4 &   & 5 & 6 & 7 & 7  \\
5 &   &   & 6 & 7 & 8 & 9 & 9   \\
\end{array}
$
\end{center}
\caption{$\gid(K_n \times K_m)$ for small $n$ and $m$}
\label{fig:table}
\end{table}

The remaining cases are handled based on the size of the second factor
relative to the first factor.  Theorem~\ref{thm:largem} presents this number if
both cliques have order at least 3 and one clique is sufficiently large compared
to the other; its proof is given in Section~\ref{sec:proofs}.

\begin{theorem} \label{thm:largem}
For positive integers $n$ and $m$ where $n \ge 3$ and $m \ge 2n$,
\[\gid(K_n \times K_m) = m-1\,.\]
\end{theorem}

In all other cases (that is, for $6 \le n \le m \le 2n-1$), the minimum cardinality
of an ID code for $K_n \times K_m$ is one of the values $\lfloor{2(n+m)/3}\rfloor$
or $\lceil{2(n+m)/3}\rceil$.  The number $\gid(K_n \times K_m)$ depends on the congruence
of $n+m$ modulo 3.   It turns out there are only 2 general cases instead of 3, but
one of them has an exception to the easily stated formula.   The exact values are
given in the following results whose proofs are given in Section~\ref{sec:proofs}.

\begin{theorem} \label{thm:0or2mod3}
Let $n$ and $m$ be positive integers such that $6 \le n \le m \le 2n-1$.
If $n + m \equiv 0 \pmod{3}$ or $n + m \equiv 2 \pmod{3}$, then
\[\gid(K_n \times K_m) = \left\lfloor\frac{2m+2n}{3}\right\rfloor\,.\]
\end{theorem}

\begin{theorem} \label{thm:2nminus5}
For a positive integer $n \ge 6$,
\[\gid(K_n \times K_{2n-5})=2n-4\,.\]
\end{theorem}

\begin{theorem} \label{thm:1mod3}
Let $n$ and $m$ be positive integers such that $6 \le n \le m \le 2n-2$
and  $m \ne 2n-5$.  If $n + m \equiv 1 \pmod{3}$, then
\[\gid(K_n \times K_m) = \left\lceil\frac{2m+2n}{3}\right\rceil\,.\]
\end{theorem}

\section{Preliminary Properties} \label{sec:properties}

In this section we prove a number of results that will be useful in proving the
minimum size of ID codes in the direct product of two complete graphs.  It will
be helpful in what follows to remember that a vertex is adjacent to
$(i,j)$ in $K_n \times K_m$ precisely when its first coordinate is
different from $i$ and its second coordinate is different from $j$. Also, recall
that we are assuming throughout that $n \le m$.

\begin{lemma} \label{lem:lowerbound}
If $C$ is an identifying code of $K_n \times K_m$, then $cs(C) \ge n-1$
and $rs(C) \ge m-1$.  In particular, $|C| \ge m-1$.
\end{lemma}
\begin{proof}
Suppose that for some $r \ne s$, $C \cap R_r = \emptyset = C \cap R_s$.
Then for any fixed $i \in [n]$, $C \cap N[(i, r)] = C- C_i = C \cap N[(i, s)]$.
Since this violates $C$ being an ID code, $K_n \times K_m$ has at most one
row disjoint from $C$.  A similar argument shows that $K_n \times K_m$ has
no more than one column disjoint from $C$. Consequently, $|C| \ge m-1$.
\end{proof}

By considering $N[x]$, the following result is obvious but useful.  We omit
its proof.

\begin{lemma} \label{lem:rowandcolseparation}
Suppose $C \subseteq V(K_n \times K_m)$ and let $x = (i, r) \in C$.  Then
$C$ separates $x$ from any $y \in (R_r \cup C_i) - \{x\}$.
\end{lemma}

Lemma~\ref{lem:rowandcolseparation} addresses separating two vertices that
belong to the same row or to the same column.  The next result concerns
vertices that are not in a common row or common column, that is, two vertices
at opposite ``corners'' of a two-row and two-column configuration in $K_n \times K_m$.

\begin{lemma} (4-Corners Property) \label{lem:4corners}
 Suppose $C$ is a dominating set of $K_n \times K_m$. For each
$(i, r), (j, s) \in K_n \times K_m$  with $i \ne j, r \ne s$, $C$ separates
$(i, r)$ and $(j, s)$ if and only if
$$C \cap(C_i \cup C_j \cup R_r \cup R_s)\not\subseteq \{i, j\} \times \{r, s\}.$$
\end{lemma}
\begin{proof}
Suppose that $i\ne j$ and $r\ne s$ and let $C_i, C_j$ and $R_r, R_s$ be
the corresponding columns and rows of $K_n \times K_m$.
Write $x=(i, r), y=(j, s), w=(i, s)$ and $z=(j, r)$ and define
\begin{eqnarray*}
A &=& C-(C \cap (C_i\cup C_j\cup R_r \cup R_s))\\
B&=& \left[C \cap (C_i\cup C_j \cup R_r \cup R_s)\right] - \{x, y, w, z\}.
\end{eqnarray*}
Then
\begin{eqnarray*}
C \cap N[x] &=& A \cup (C \cap \{x, y\}) \cup (C \cap ((R_s \cup C_j)-\{x,y,w, z\}))\\
C \cap N[y] &=& A \cup (C \cap \{x, y\}) \cup (C \cap ((R_r \cup C_i)-\{x,y,w, z\}))
\end{eqnarray*}
Therefore, $C$ separates $x$ and $y$ if and only if at least one of the two disjoint sets
$C \cap ((R_s \cup C_j)-\{x,y,w, z\})$ or $C \cap ((R_r \cup C_i)-\{x,y,w,z\})$ is non-empty.
Since $B$ is the union of these 2 sets, it follows that
$C$ separates $x$ and $y$ if and only if $B\not=\emptyset$, or equivalently if and only if
\[ C \cap(C_i \cup C_j \cup R_r \cup R_s)\not\subseteq \{i, j\} \times \{r, s\}\,.\]
\end{proof}

\noindent We will say that a dominating set $D$ of $K_n \times K_m$ has the {\it 4-corners property
with respect to columns $C_i$, $C_j$ and rows $R_r$, $R_s$} if
\[D \cap (C_i \cup C_j \cup R_r \cup R_s)\not\subseteq \{i, j\} \times \{r, s\}\,.\]

\noindent Hence, if a dominating set $D$ of $K_n \times K_m$ is an ID code, then $D$ has the
4-corners property with respect to every pair of columns and every pair of rows.
Each of the next three results follows immediately from this fact.

\begin{cor} \label{cor:3results1}
If $C$ is an identifying code of $K_n \times K_m$, then
$C$ has no more than one isolated codeword.
\end{cor}

\begin{cor} \label{cor:3results2}
Let $C$ be an identifying code of $K_n \times K_m$.
If $cs(C) = n-1$, then there does not exist a column $C_j$ such that
$C \cap C_j = \{u,v\}$ where both $u$ and $v$ are row-isolated.
Similarly, there is no row $R_r$ containing exactly two codewords each of
which is column-isolated if $rs(C) = m-1$.
\end{cor}

\begin{cor} \label{cor:3results3}
If $C$ is an identifying code of $K_n \times K_m$ such that $cs(C) = n-1$
and $rs(C) = m-1$, then $C$ has no isolated codeword.
\end{cor}

The next two results will be used to construct ID codes thereby providing
an upper bound for $\gid(K_n \times K_m)$.  Which one is used will depend on
the congruence of $n+m$ modulo 3.

\begin{proposition} \label{prop:fullspan}
Let $C \subset  V(K_n \times K_m)$. Then $C$ is an identifying code of $K_n \times K_m$
if it satisfies the following conditions.
\begin{enumerate}
\item[(1)] There exist $1 \le n_1 < n_2 < n_3 \le n$ and $1 \le m_1 < m_2 < m_3 \le m$  such that\\
 $(n_1, m_1), (n_2,m_2), (n_3, m_3) \in C$;
\item[(2)] Each $v \in C$ is either row-isolated or column-isolated;
\item[(3)] $rs(C) = m$ and $cs(C) = n$; and
\item[(4)] $C$ contains at most one isolated vertex.
\end{enumerate}
\end{proposition}

\begin{proof}
Assume $C$ is as specified.  For ease of reference we denote the
graph $K_n \times K_m$ by $G$ throughout this proof. By the first assumption
above it follows immediately that $C$ dominates $G$ since
$\{(n_1,m_1), (n_2,m_2), (n_3,m_3)\}$ does.

We need only show that $C$ separates every pair $x, y$ of distinct vertices.  First assume that
$x$ and $y$ are in the same column.  If $x$ or $y$ belongs to $C$, then
Lemma~\ref{lem:rowandcolseparation} shows that $C$ separates them.  If neither is in $C$,
then by our assumptions $rs(C) = m$ and $cs(C) = n$
we can choose a vertex $z\in C$ from the same row as $x$.  This vertex $z$ separates $x$ and $y$.
Similarly, $C$ separates any two vertices belonging to a common row.

Now, assume $x=(i,r)$ and $y=(j,s)$ where  $1 \le i< j \le n$ and $1 \le r<s \le m$.
Any $v = (k, t) \in C$ that is not isolated in $C$
is row-isolated or column-isolated but not both, and it follows that either $|C \cap C_k| \ge 2$
or $|C \cap R_t| \ge 2$.
\begin{enumerate}
\item[(a)] Suppose $x \in C$ but is not isolated in $C$. Then as above, either $|C \cap C_i| \ge 2$ or
$|C \cap R_r|\ge 2$. Assume without loss of generality that
$|C\cap C_i|\ge 2$. Then either $(i,s) \in C$ or there exists $1 \le t \le m$ where $t \not\in \{r, s\}$
and $(i, t) \in C$. In the first case where we have $(i, s) \in C$, it follows that $(i,s)$ is
row-isolated, and thus $y \not\in C$. However,
each column of $G$ is in the column span of $C$ so  there exists $1 \le p \le m$ where
$p \not\in \{ r, s\}$ and $(j, p) \in C$ since $(i, r)$ and $(i, s)$ are row-isolated. Thus
$(j, p) \in C \cap N[x]$ but $(j, p) \not\in C \cap N[y]$ and hence $C$ separates $x$ and $y$.
On the other hand, if there exists $1 \le t \le m$ where $t \not\in \{r, s\}$ and $(i, t) \in C$,
then $(i, t) \in C \cap N[y]$ but $(i, t) \not\in C \cap N[x]$ and hence $C$ separates $x$ and $y$.
If we had instead assumed that $|C \cap R_r| \ge2$, that is we had assumed $x$ is column-isolated
and not row-isolated, then a similar argument shows that $C$ separates $x$ and $y$.

\item[(b)] Suppose $x\in C$ and is isolated in $C$. Since $x$ is both row-isolated and column-isolated
$C = C \cap N[x]$. First assume that $y \not\in C$. Since $C_j$ is in the column span
 of $C$,  there exists $1 \le t \le m$ with $t \not\in \{ r, s\}$ such that $(j, t) \in C$,
and $(j, t)$  separates $x$ and $y$. On the other hand, if $y \in C$
 then either $|C \cap C_j| \ge 2$ or $|C \cap R_s|\ge2$ since $y$ is not isolated. In either case,
 $C \cap N[y] \ne C$ and therefore $C$ separates $x$ and $y$.

\item[(c)] Suppose $x, y \in V(G) - C$. If we assume that $C$ does not separate $x$ and $y$, then
because each row of $G$ is in the row span of $C$ and each column of $G$ is in the column span of
$G$, it follows that
$$C \cap (C_i \cup C_j \cup R_r \cup R_s) = \{(i,s), (j, r)\}\,.$$
Thus by definition, both $(i, s)$ and $(j, r)$ are isolated in $C$, contradicting the
fourth assumption. Hence, $C$ separates $x$ and $y$.
\end{enumerate}
Therefore $C$ separates every pair of distinct vertices, and thus $C$ is an ID code of $K_n \times K_m$.
\end{proof}

\begin{proposition} \label{prop:partialspan}
Let $C \subset  V(K_n \times K_m)$.  Then $C$ is an identifying code of $K_n \times K_m$
if it satisfies the following conditions.
\begin{enumerate}
\item[(1)] There exist $1 \le n_1 < n_2 < n_3 \le n$ and $1 \le m_1 < m_2 < m_3 \le m$ such that\\
$(n_1, m_1), (n_2,m_2), (n_3, m_3) \in C$;
\item[(2)] Every $v \in C$ is either row-isolated or column-isolated;
\item[(3)] $rs(C) = m-1$ and $cs(C) = n$;
\item[(4)] $C$ contains at most one isolated vertex; and
\item[(5)] If $R_r$ has the property that every $v \in C \cap R_r$ is column-isolated
           but not row-isolated, then $|C \cap R_r| \ge 3$.
\end{enumerate}
\end{proposition}
\begin{proof}
As in the proof of Proposition~\ref{prop:fullspan} we see that $C$ dominates $G=K_n\times K_m$.

We show that $C$ separates every pair $x, y$ of distinct vertices in $G$.
Let $R_r$ be the row not in the row span of $C$. Notice that $V(G) - R_r \cong K_n \times K_{m-1}$
and that $C$ satisfies the hypotheses of Proposition~\ref{prop:fullspan} when considered
as a subset of $V(G) - R_r$.
Thus $C$ separates $x, y$ if neither is in $R_r$, and so we may assume that $x \in R_r$,
say $x=(i,r)$.
\begin{enumerate}
\item[(a)] First assume that $y = (j, r)$ with $i \ne j$.  Since
$cs(C) = n$, there exists $1 \le s \le m$
such that $r \ne s$ and $(i, s) \in C$. This vertex $(i, s)$  separates $x$ and $y$.
Next, assume that $y = (i, t)$ for some $1 \le t \le m$ with $t \ne r$.
If $y \in C$ then $y$ separates $x$ and $y$.
However if $y \not\in C$, then since each row of $G$, other than $R_r$, is in the row span of $C$
there exists $1 \le j \le n$ with $i \ne j$ such that $(j, t) \in C$. It follows that
$(j, t)$ separates $x$ and $y$.
\item[(b)] Next, assume that $y = (j, s)$ where $i \ne j$ and $r \ne s$. If we assume that $C$ does
not separate $x$ and $y$, then $C$ does not satisfy the 4-Corners Property with respect to
columns $C_i$, $C_j$ and rows $R_r$, $R_s$.  In addition, since $R_r$ is not in the row span of $C$
$$C \cap (C_i \cup C_j \cup R_r \cup R_s)\subseteq \{(i, s), (j, s)\}\,.$$
Since both $C_i$ and $C_j$ are in the row span of $C$, it follows that
$C \cap (C_i \cup C_j \cup R_r \cup R_s)= \{(i, s), (j, s)\}$. This means that $R_s$
contains exactly two members of $C$ and they are both column-isolated, contradicting
one of the assumptions. Hence, this case cannot occur either, and it follows that $C$ separates $x$ and $y$.
\end{enumerate}
Therefore, $C$ is an ID code of $K_n \times K_m$.
\end{proof}

\section{Proofs of Main Results}  \label{sec:proofs}

In this section we prove all of our main results.  The general strategy will be to construct
an ID code of the claimed optimal size (by employing
Propositions~\ref{prop:fullspan} and \ref{prop:partialspan}) and prove
the given direct product has no smaller ID code.

\setcounter{theorem}{0}
We treat the smallest case first.
\begin{theorem}
For any positive integer $m \ge 5$, $\gid(K_2 \times K_m)=m-1$. In addition,
if $3 \le m \le 4$, $\gid(K_2\times K_m)=m$.
\end{theorem}
\begin{proof}
If $C$ is any ID code of $K_2\times K_3$, then $rs(C)\ge 2$.  No subset of 2 elements
in different rows dominates $K_2 \times K_3$, and so $\gid(K_2 \times K_3) \ge 3$.  It is
easy to check that $\{(1,1), (1,2), (1,3)\}$ is an ID code.  A similar argument shows
that $\gid(K_2 \times K_4)=4$.

If $m \ge 5$, it follows from Lemma~\ref{lem:lowerbound} that $\gid(K_2\times K_m) \ge m-1$,
and it is easily checked that $\{(1,1),(1,2)\} \cup \{(2,r)\,|\, 3\le r \le m-1\}$
is an ID code.
\end{proof}

Now we turn our attention to the case when the first factor has order at least three and the
second factor is sufficiently larger than the first.

\begin{theorem}
For positive integers $n$ and $m$ where $n \ge 3$ and $m \ge 2n$,
\[\gid(K_n \times K_m) = m-1\,.\]
\end{theorem}

\begin{proof}
Consider the set
$$D = \{(i, 2i-1), (i, 2i) \,|\, i \in [n-1]\}\,\, \cup \,\,\{(n, j)\, |\, 2n-1 \le j \le m-1\}.$$
Notice that each $v$ in $D$ is row-isolated but not column-isolated, $rs(D)=m-1$ and $cs(D) = n.$
Furthermore, $(1,1),(2,3)$ and $(3,5) \in D$. Thus Proposition~\ref{prop:partialspan} guarantees
that $D$ is an ID code and Lemma~\ref{lem:lowerbound} gives the desired result.
\end{proof}

\noindent We now focus on direct products of the form $K_n \times K_m$ where
$6 \le n \le m \le 2n-1$ and prove that in all cases

\begin{equation} \label{eqn:mainresult}
\left\lfloor \frac{2m+2n}{3}\right\rfloor \le \gid(K_n \times K_m) \le \left\lceil
\frac{2m+2n}{3} \right\rceil\,.
\end{equation}

\noindent For the remainder of this paper, when considering any ID code $C$ of
$G=K_n \times K_m$ we define $A_c = \{ v \in C\,|\,v \text{ is row-isolated in }C\}$
and $B_c = \{ v \in C\,|\, v \text{ is column-isolated in }C\}$.
Let $|A_c| = x$ and
let $p$ denote the number of columns $C_i$ of $G$ such that $|C \cap C_i| \ge 2$ and
$C \cap C_i \subseteq A_c$. Similarly, let $|B_c|=y$ and let $q$ represent the number
of rows $R_r$ of $G$ such that $|C \cap R_r| \ge 2$ and $C \cap R_r \subseteq B_c$.
 Notice that $C$ contains at most one isolated codeword, in which case
$|A_c \cap B_c| = 1$. Otherwise, $A_c \cap B_c = \emptyset$. Moreover, we always have
$|C| \ge |A_c \cup B_c| \ge x+y-1$.

\begin{theorem} \label{thm:mod0or2}
If $n$ and $m$ are positive integers such that   $6 \le n \le m \le 2n-1$
and $n + m \equiv 0 \pmod{3}$ or $n + m \equiv 2 \pmod{3}$, then
$$\gid(K_n \times K_m) = \left\lfloor\frac{2m+2n}{3}\right\rfloor.$$
\end{theorem}

\begin{proof}
Suppose $C$ is an ID code of $G=K_n \times K_m$ such that
$|C| \le \left\lfloor\frac{2n+2m}{3}\right\rfloor - 1$.
We consider 4 cases based on the possible values of $cs(C)$ and $rs(C)$.
\begin{enumerate}
\item[Case 1] Suppose $cs(C) = n$ and $rs(C) = m$.\\
Since $|B_c|=y$,  $|C-B_c| \ge 2(n-y)$ which implies $|C| \ge  2n-y$. Then
$\frac{2m+2n}{3}-1 \ge |C| \ge 2n-y$, and it follows that
$y \ge \frac{4n-2m}{3} + 1$.  Similarly, we get $x \ge \frac{4m-2n}{3}+1$.
Together these imply that
\[\frac{2m+2n}{3}-1 \ge |C| \ge x+y-1 \ge \frac{2m+2n}{3} + 1\,.\]
This is clearly a contradiction, and hence no such $C$ exists with $cs(C) = n$ and
$rs(C) = m$.

\item[Case 2] Suppose $cs(C) = n-1$ and $rs(C) = m$.\\
Note that since each codeword in $B_c$ is column-isolated and $cs(C) = n-1$, there exist
at least 2 codewords in each of the remaining $n-1-y$ columns  disjoint from
the column span of $B_c$. However, Corollary~\ref{cor:3results2} guarantees that
$|C\cap C_j| \ge 3$ for any column $C_j$ for which $|C \cap C_j| \ge 2$ and
$C \cap C_j \subseteq A_c$. Since $p$ represents the number of such
columns, $|C-B_c| \ge 2(n-1-y-p) + 3p = 2n - 2 - 2y + p$. So $|C| \ge 2n - 2 - y + p$.
Consequently, $y \ge \frac{4n-2m}{3} - 1 + p$.

Similarly, since $|A_c|=x$ and  $rs(C) = m$, $|C-A_c| \ge 2(m-x)$
which implies $|C| \ge 2m-x$. From Case 1 we see
that this gives $x \ge \frac{4m-2n}{3} +1$. Moreover, $|C| \ge x + y -1$ so that
\[\frac{2m+2n}{3} - 1 \ge |C| \ge x+y-1 \ge \frac{2m+2n}{3} + p -1\,.\]

Thus $p \le 0$. Hence
$p=0$, and we have equality in the above so that
$$\left\lfloor\frac{2m+2n}{3}\right\rfloor-1 = |C| = x + y-1. $$
It follows that $C = A_c \cup B_c$. If there exists $v \in C-B_c$, say $v\in C_i$,
then by Corollary~\ref{cor:3results2}, $|C\cap C_i| \ge 3$.  However, this
contradicts $p=0$ since each codeword is either row-isolated or column-isolated.
 Consequently, $m=rs(C)\le |C|=|B_c|\le n-1\le m-1$.
 This contradiction shows that this case cannot occur.

\item[Case 3] Suppose $cs(C) = n$ and $rs(C) = m-1$.\\

If we interchange the roles of rows and columns in Case 2, then we are
led to $q=0$ and
\[\left\lfloor\frac{2m+2n}{3}\right\rfloor-1 = |C| = x + y-1\,.\]

Thus  $C = A_c \cup B_c$. On the other hand, since $cs(C)=n$ it follows
as in Case 1 that
$$y \ge \frac{4n-2m}{3} + 1 \ge \frac{4n -2(2n-1)}{3} + 1 = \frac{5}{3}\,.$$
Since $y$ is integral we conclude by Corollary~\ref{cor:3results2} that $q \ge 1$.
This contradiction shows that this case cannot occur.

\item[Case 4] Suppose that $cs(C) = n-1$ and $rs(C) = m-1$.\\
 From Case 2 and Case 3, we see that
 $$y \ge \frac{4n-2m}{3} - 1 + p \quad \text{and} \quad x \ge \frac{4m-2n}{3} - 1 + q\,.$$

 Since $cs(C) = n-1$ and $rs(C) = m-1$, it follows from Corollary~\ref{cor:3results3}
 that $C$ does not contain an isolated vertex.  It follows that
\[ \frac{2m+2n}{3} - 1 \ge |C| \ge x + y \ge = \frac{2m+2n}{3} -2 + p + q\,.\]
Hence $p+q \le 1$.

Suppose $p=1$.  Then we have equality throughout the above inequality, and thus
$C=A_c \cup B_c$.  Suppose there exists $v \in B_c$, say $v\in R_r$. Since $q=0$
and there are no isolated codewords, it follows that $C$ contains another
codeword $u$ in $R_r$ that is not column-isolated.  But $u \not\in A_c\cup B_c$
which is a contradiction.  Therefore, $C=A_c$. Since $p=1$ we are led to conclude
that $cs(C)=1$, another contradiction.

To show that $q=1$ is not possible we simply interchange the roles of $A_c$
and $B_c$ in the above.

Finally, suppose $p=0=q$.

Since $p=0$, any column that contains a row-isolated codeword would also have to contain a
codeword  that is not row-isolated. Since there can exist at
most one of these to guarantee $|C| \le \left\lfloor \frac{2m+2n}{3}\right\rfloor -1$, then
there is a column $C_i$  such that $A_c \subseteq C_i$ and for some $r$,
$(i,r)\in C-(A_c \cup B_c)$. Similarly, since $q=0$, if there exists a row containing
a column-isolated codeword, then that row contains a codeword that is not
column-isolated.  Since $|C-(A_c \cup B_c)| \le 1$, such a codeword must be
$(i,r)$.  This implies that $\frac{2m+2n}{3}-1 \ge |C| \ge m-1+n-2$, and this
implies that $n+m \le 6$, a contradiction.
\end{enumerate}

Therefore, every ID code of $K_n\times K_m$ has cardinality at least
$\lfloor\frac{2m+2n}{3}\rfloor$.

An application of Proposition~\ref{prop:fullspan} shows that the following sets are
ID codes of cardinality $\lfloor\frac{2m+2n}{3}\rfloor$ and finishes the proof.

If $n + m \equiv 0 \pmod{3}$, let
$$D_1 = \{(i, 2i-1), (i, 2i)|1 \le i \le a\} \cup \{(a + 2j-1, 2a + j), (a + 2j, 2a + j) | 1 \le j \le b\}\,,$$
where $a = \frac{2m-n}{3}$ and $b=\frac{2n-m}{3}$.
For $n+m \equiv 2 \pmod{3}$ but $m\ne 2n-1$, let $a=\frac{2m-n-1}{3}$, $b = \frac{2n-m-1}{3}$, and
$$D_2 = \{(i, 2i-1), (i, 2i) \big| 1 \le i \le a\} \cup \{(a + 2j-1,2a+j),
(a+2j, 2a +j)\big|1 \le j \le b\}  \cup \{(n,m)\}\,.$$
Finally, if $m=2n-1$, let
$$D_3 = \{(i, 2i-1), (i, 2i) | i \in [n-1]\}\,\, \cup \,\, \{(n, 2n-1)\}.$$
\end{proof}

The following figure illustrates ID codes of optimal order for several
of the cases of Theorem~\ref{thm:mod0or2}.  The vertices of the  direct products
in the figure are represented but the edges are omitted for clarity.  Recall
that columns are vertical and rows are horizontal.  Solid
vertices indicate the members of an optimal ID code in each case.

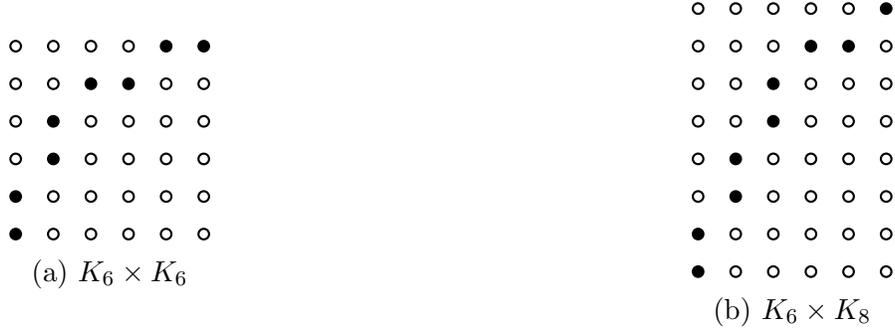
\begin{figure}[ht!]
\begin{subfigure}[]{.5\textwidth}
\centering
\begin{tikzpicture}[scale=1.0,style=thick]
\def\vr{2pt} 
\draw (0,0) [fill=black] circle (\vr); \draw (.5,0) [fill=white] circle (\vr);
\draw (1,0) [fill=white] circle (\vr); \draw (1.5,0) [fill=white] circle (\vr);
\draw (2,0) [fill=white] circle (\vr); \draw (2.5,0) [fill=white] circle (\vr);
\draw (0,.5) [fill=black] circle (\vr); \draw (.5,.5) [fill=white] circle (\vr);
\draw (1,.5) [fill=white] circle (\vr); \draw (1.5,.5) [fill=white] circle (\vr);
\draw (2,.5) [fill=white] circle (\vr); \draw (2.5,.5) [fill=white] circle (\vr);
\draw (0,1) [fill=white] circle (\vr); \draw (.5,1) [fill=black] circle (\vr);
\draw (1,1) [fill=white] circle (\vr); \draw (1.5,1) [fill=white] circle (\vr);
\draw (2,1) [fill=white] circle (\vr); \draw (2.5,1) [fill=white] circle (\vr);
\draw (0,1.5) [fill=white] circle (\vr); \draw (.5,1.5) [fill=black] circle (\vr);
\draw (1,1.5) [fill=white] circle (\vr); \draw (1.5,1.5) [fill=white] circle (\vr);
\draw (2,1.5) [fill=white] circle (\vr); \draw (2.5,1.5) [fill=white] circle (\vr);
\draw (0,2) [fill=white] circle (\vr); \draw (.5,2) [fill=white] circle (\vr);
\draw (1,2) [fill=black] circle (\vr); \draw (1.5,2) [fill=black] circle (\vr);
\draw (2,2) [fill=white] circle (\vr); \draw (2.5,2) [fill=white] circle (\vr);
\draw (0,2.5) [fill=white] circle (\vr); \draw (.5,2.5) [fill=white] circle (\vr);
\draw (1,2.5) [fill=white] circle (\vr); \draw (1.5,2.5) [fill=white] circle (\vr);
\draw (2,2.5) [fill=black] circle (\vr); \draw (2.5,2.5) [fill=black] circle (\vr);
\end{tikzpicture}
\caption{$K_6\times K_6$}
\end{subfigure}~
\begin{subfigure}[]{.5\textwidth}
\centering
\begin{tikzpicture}[scale=1.0,style=thick]
\def\vr{2pt} 
\draw (0,0) [fill=black] circle (\vr); \draw (.5,0) [fill=white] circle (\vr);
\draw (1,0) [fill=white] circle (\vr); \draw (1.5,0) [fill=white] circle (\vr);
\draw (2,0) [fill=white] circle (\vr); \draw (2.5,0) [fill=white] circle (\vr);
\draw (0,.5) [fill=black] circle (\vr); \draw (.5,.5) [fill=white] circle (\vr);
\draw (1,.5) [fill=white] circle (\vr); \draw (1.5,.5) [fill=white] circle (\vr);
\draw (2,.5) [fill=white] circle (\vr); \draw (2.5,.5) [fill=white] circle (\vr);
\draw (0,1) [fill=white] circle (\vr); \draw (.5,1) [fill=black] circle (\vr);
\draw (1,1) [fill=white] circle (\vr); \draw (1.5,1) [fill=white] circle (\vr);
\draw (2,1) [fill=white] circle (\vr); \draw (2.5,1) [fill=white] circle (\vr);
\draw (0,1.5) [fill=white] circle (\vr); \draw (.5,1.5) [fill=black] circle (\vr);
\draw (1,1.5) [fill=white] circle (\vr); \draw (1.5,1.5) [fill=white] circle (\vr);
\draw (2,1.5) [fill=white] circle (\vr); \draw (2.5,1.5) [fill=white] circle (\vr);
\draw (0,2) [fill=white] circle (\vr); \draw (.5,2) [fill=white] circle (\vr);
\draw (1,2) [fill=black] circle (\vr); \draw (1.5,2) [fill=white] circle (\vr);
\draw (2,2) [fill=white] circle (\vr); \draw (2.5,2) [fill=white] circle (\vr);
\draw (0,2.5) [fill=white] circle (\vr); \draw (.5,2.5) [fill=white] circle (\vr);
\draw (1,2.5) [fill=black] circle (\vr); \draw (1.5,2.5) [fill=white] circle (\vr);
\draw (2,2.5) [fill=white] circle (\vr); \draw (2.5,2.5) [fill=white] circle (\vr);
\draw (0,3) [fill=white] circle (\vr); \draw (.5,3) [fill=white] circle (\vr);
\draw (1,3) [fill=white] circle (\vr); \draw (1.5,3) [fill=black] circle (\vr);
\draw (2,3) [fill=black] circle (\vr); \draw (2.5,3) [fill=white] circle (\vr);
\draw (0,3.5) [fill=white] circle (\vr); \draw (.5,3.5) [fill=white] circle (\vr);
\draw (1,3.5) [fill=white] circle (\vr); \draw (1.5,3.5) [fill=white] circle (\vr);
\draw (2,3.5) [fill=white] circle (\vr); \draw (2.5,3.5) [fill=black] circle (\vr);
\end{tikzpicture}
\caption{$K_6\times K_8$}
\end{subfigure}
\caption{Examples of ID codes when $n+m \equiv 0,2\pmod{3}$}
\label{fig:IDCodes1}
\end{figure}

For a fixed $n \ge 6$ the lone exception to the formula $\lceil\frac{2m+2n}{3}\rceil$ for 
$\gid(K_n \times K_m)$ where $n \le m \le 2n-2$ and $n+m$ congruent to 1 modulo 3
is the instance $m=2n-5$.  We now prove Theorem~\ref{thm:2nminus5} which shows
the correct value is $\lfloor\frac{2(2n-5)+2n}{3}\rfloor$.  We restate it here for 
convenience.

\begin{theorem}
For a positive integer $n \ge 6$,
\[\gid(K_n \times K_{2n-5})=2n-4\,.\]
\end{theorem}

\begin{proof}
Assume there exists an ID code $C$ for $K_n \times K_{2n-5}$ such that
$|C| \le 2n-5$. Since $rs(C) \ge 2n-6$, we consider the following 2 cases.
\begin{enumerate}
\item[Case 1] Suppose that $rs(C) = 2n-6.$\\
Since each codeword in $A_c$ is row-isolated and $rs(C) = 2n-6$, there exist at
least 2 codewords in each of the remaining $2n-6-x$ rows disjoint from the row
span of $A_c$. However, Corollary~\ref{cor:3results2} guarantees that
$|C \cap R_r| \ge 3$ for any row $R_r$
where $C \cap R_r \subseteq B_c$. Since $q$ represents the number of these rows,
$|C-A_c| \ge 2(2n-6-x-q)+3q$ which implies $|C| \ge 4n-12 -x+q$. Consequently,
$2n-5 \ge 4n-12-x+q$ which implies $x \ge 2n-7+q.$

Similarly, since $cs(C) \ge n-1$ and each codeword in $ B_c$ is column-isolated,
there exist at least 2 codewords of $C$ in each of the remaining $n-1-y$ columns
disjoint from the column span of $B_c$. Thus $|C-B_c| \ge 2(n-1-y)$ which implies that
$|C| \ge 2n-2-y$. Therefore, $ y \ge 3.$ It follows that
$$2n-5 \ge |C| \ge x + y -1 \ge 2n-5 + q.$$
Thus, $q =0$. Moreover,  we have equality in the above and therefore $C = A_c \cup B_c$.
 On the other hand, $y \ge 3$ and only one of these column-isolated codewords can be
 isolated. Consequently, $q \ge 1$ since each codeword of $C$ is either row-isolated
 or column-isolated, a contradiction.
\item[Case 2] Suppose $rs(C) = 2n-5$. \\
Using a similar argument as in Case 1, we have $|C-A_c| \ge 2(2n-5-x)$ which implies
$|C| \ge 4n-10-x$. This implies $2n-5 \ge |C| \ge x \ge 2n-5.$
Therefore, it follows that $C = A_c$, and thus $cs(C)=cs(A_c) \le \frac{2n-6}{2} + 1 = n-2$,
a contradiction to Lemma~\ref{lem:lowerbound}.
\end{enumerate}
Therefore, no such identifying code $C$ exists with $|C| \le 2n-5$. It follows that
$\gid(G) \ge 2n-4$.

An application of Proposition~\ref{prop:partialspan} shows that the
set
$$D = \{(i, 2i-1), (i, 2i) \big| 1 \le i \le n-4\} \cup \{(n-3, 2n-7),(n-2,2n-7),
(n-1, 2n-7),(n, 2n-6)\}\,,$$
is an ID code of $K_n \times K_{2n-5}$ of cardinality $2n-4$.

\end{proof}

\begin{theorem} 
Let $n$ and $m$ be positive integers such that $6 \le n \le m \le 2n-2$
and  $m \ne 2n-5$.  If $n + m \equiv 1 \pmod{3}$, then
\[\gid(K_n \times K_m) = \left\lceil\frac{2m+2n}{3}\right\rceil\,.\]
\end{theorem}

\begin{proof}
First, notice that $\lceil\frac{2m+2n}{3}\rceil=\frac{2m+2n+1}{3}$.  Assume
that there exists an ID code $C$ for $K_n \times K_m$ such that $|C| \le \frac{2n+2m+1}{3} - 1$.
We again consider 4 cases based on the possible values of $cs(C)$ and $rs(C)$.
\begin{enumerate}
\item[Case 1] Suppose $cs(C) = n$ and $rs(C) = m$.\\
Using reasoning similar to that in Case 1 of the proof of Theorem~\ref{thm:mod0or2}
we get $y \ge \frac{4n-2m+2}{3}$, and $x \ge \frac{4m-2n+2}{3}$.
On the other hand, we know $|C| \ge x+y -1$. Consequently,
$\frac{2m+2n+1}{3}-1 \ge x+y-1 \ge \frac{2m+2n+1}{3}$, which is
clearly a contradiction.

\item[Case 2] Suppose $cs(C) = n-1$ and $rs(C) = m$.\\
Since $|B_c|=y$  and $cs(C) = n-1$, there exist
at least 2 codewords in each of the remaining $n-1-y$ columns that are disjoint from
the column span of $B_c$. However, Corollary~\ref{cor:3results2} guarantees
$|C\cap C_j| \ge 3$ for any such column $C_j$ where
$C \cap C_j \subseteq A_c$. Since $p$ represents the number of
these columns, then $|C-B_c| \ge 2(n-1-y-p) + 3p = 2n - 2 - 2y + p$. As a
result it follows that  $y \ge \frac{4n-2m-4}{3}  + p$.

Similarly, since $rs(C) = m$ and $x=|A_c|$ we get $|C-A_c| \ge 2(m-x) $ which
implies $|C| \ge 2m-x$. As in  Case 1 it follows that  $x \ge \frac{4m-2n+2}{3}$.
This yields
\[ \frac{2m+2n+1}{3} - 1 \ge |C| \ge x+y-1 \ge \frac{2m+2n+1}{3} + p -2\,.\]

Thus $p \le 1$. Assume first that $p=1$. Then we have equality
in the above  and thus $C = A_c \cup B_c$,
$y = \frac{4n-2m-1}{3}$ and $x = \frac{4m-2n+2}{3}.$ Furthermore,
$C$ contains an isolated codeword, call it $v$. Since $p=1$, there exists a column $C_i$
 such that $A_c - \{v\} = C \cap C_i$. It follows that $cs(A_c) = 2$. On the
other hand, $cs(C)=n-1$ so $B_c - \{v\}$ spans the remaining $n-3$ columns. Therefore,
$n-3 = \frac{4n-2m-1}{3}-1$ which implies $m<n$, a contradiction.

Therefore, $p=0$.  First assume that $C$ contains no isolated codeword.  Then necessarily
$C=A_c\cup B_c$.  As in the proof of Case 2 of Theorem~\ref{thm:mod0or2}
we arrive at a contradiction, and hence $C$ does contain an isolated codeword, say $v$.  Because $p=0$,
any column that contains a row-isolated codeword other than $v$ would also
have to contain a codeword that is not row-isolated. Note that
 $x \ge \frac{4m-2n+2}{3}\ge 5$, and hence there exists a
 column $C_i$  such that $A_c - \{v\} \subseteq C \cap C_i$.
 In addition there exists a codeword $(i,r)$ that is neither row-isolated
 nor column-isolated.  This means $C=A_c\cup B_c \cup \{(i,r)\}$ and so
 $y=\frac{4n-2m-4}{3}$.   It follows that
 $cs(A_c) = 2$.  On the other hand, $cs(C)=n-1$ so $B_c-\{v\}$ spans
 the remaining $n-3$ columns.  Therefore,
$n-3 = \frac{4n-2m-4}{3}-1$ which implies $2m=n+2$, a contradiction.

\item[Case 3] Suppose $cs(C) = n$ and $rs(C) = m-1$.\\

Since $|A_c|=x$ and $rs(C) = m-1$,  there exist
at least 2 codewords  in each of the remaining $m-1-x$ rows disjoint from the
row span of $A_c$. However, Corollary~\ref{cor:3results2} guarantees
$|C \cap R_r| \ge 3$ for any such row $R_r$
where $C \cap R_r \subseteq B_c$. Since $q$ represents the number of these rows,
then $|C-A_c| \ge 2(m-1-x-q) + 3q= 2m-2-2x +q$. This implies that
$x \ge \frac{4m-2n-4}{3} + q$.
Similarly, since $cs(C) = n$ and $|B_c| = y$ we get $|C-B_c| \ge 2(n-y)$ which
implies $|C| \ge 2n - y$. As in Case 1 it follows that $y \ge \frac{4n-2m+2}{3}$.
Consequently,
$$\frac{2m+2n+1}{3}-1 \ge |C| \ge x+y-1 \ge  \frac{2m+2n+1}{3} + q -2.$$

Thus $q \le 1$. Assume first that $q=1$. Then we have equality in the above and
thus $C = A_c \cup B_c$, $y = \frac{4n-2m+2}{3}$ and $x = \frac{4m-2n-1}{3}.$
Furthermore, $C$ contains an isolated codeword, call it $v$. Since $q=1$, there
exists a row $R_r$ such that $B_c - \{v\} = C \cap R_r$. Thus $rs(B_c) = 2$. On
the other hand, $rs(C) = m-1$ so $A_c - \{v\}$ spans the remaining $m-3$ rows.
Therefore, $m-3 = \frac{4m-2n-1}{3}-1$ which implies $m = 2n-5$, a contradiction.

Therefore, $q=0$. First assume $C$ contains no isolated codeword. Then necessarily
$C = A_c \cup B_c$ and since $q=0$, it follows that  $C = A_c$.
Since $cs(C) = n$ and no isolated codeword exists, it follows that
$|C| \ge 2n$. Therefore, $\frac{2m+2n+1}{3}-1 \ge 2n$ which implies $m \ge 2n+1$,
a contradiction. So $C$ contains an isolated codeword, call it $v$.

Because $q=0$, any row that contains a column-isolated codeword other than $v$ would
also have to contain a codeword that is not column-isolated. Note that
$y \ge \frac{4n-2m+2}{3} \ge \frac{4n-2(2n-2)+2}{3} = 2$ and hence there exists a
row $R_r$ such that $B_c - \{v\} \subset C \cap R_r$. In addition, there exists a
codeword $(i, r) \in C \cap R_r$ that is not column-isolated. Thus
$C = A_c \cup B_c \cup \{(i, r)\}$ and so $x = \frac{4m-2n-4}{3}$. It follows that
$rs(B_c) = 2$. On the other hand, $rs(C) = m-1$ so $A_c - \{v\}$ spans the remaining
$m-3$ rows. Therefore $m-3 = \frac{4m-2n-4}{3}-1$ which implies $m=2n-2$. However,
in this specific case $x = 2n-4$ and $y=2$. Consequently,
$n = cs(C) \le \frac{2n-5}{2} + 2 = n - \frac{1}{2}$, a contradiction.

\item[Case 4] Suppose that $cs(C) = n-1$ and $rs(C) = m-1$.\\
 From Case 2 and Case 3, we see that
 $$y \ge \frac{4n-2m-4}{3}  + p \quad \text{and}\quad x \ge \frac{4m-2n-4}{3} + q.$$
  Since $cs(C) = n-1$ and $rs(C) = m-1$, it follows from
  Corollary~\ref{cor:3results3} that $C$ does not
  contain an isolated codeword. Thus
$$\frac{2m+2n+1}{3} - 1 \ge |C| \ge x + y \ge  \frac{2m+2n+1}{3} -3 + p + q.$$

Hence $p+q \le 2$.
\begin{enumerate}
\item[(i)] Suppose that $p=0$. Then for each column $C_i$  where $A_c \cap C_i \ne \emptyset$,
there will exist another codeword in $C_i$ that is not row-isolated. To guarantee that
$\frac{2m+2n+1}{3} -1 \ge |C|$, $C$ contains at most 2 such codewords. Therefore,
$cs(A_c) \le 2$.
 If $cs(A_c)=2$, then $y = \frac{4n-2m-4}{3}$ and it
follows that
$$n-1 = cs(C) = cs(A_c) + cs(B_c) = 2 + \frac{4n-2m-4}{3}.$$
This implies
$m < n$, a contradiction, and thus $cs(A_c) < 2$. On the other hand,\\
$x \ge \frac{4m-2n-4}{3} +q \ge  \frac{8}{3}$.  Hence $C$ contains a codeword that is
neither row-isolated or column-isolated which yields $cs(A_c) = 1$. To guarantee
$\frac{2m+2n+1}{3}-1 \ge |C|$, it must be the case that $y \le \frac{4n-2m-4}{3} + 1$.
Here again we see $cs(C) = cs(A_c) + cs(B_c) = 2 + \frac{4n-2m-4}{3}$,
which we already know to be a contradiction. Thus, $p \ne 0$.

\item[(ii)]
Suppose that $q=0$. Then for each row $R_r$ where $B_c \cap R_r \ne \emptyset$,
there will exist another codeword in $R_r$ that is not column-isolated. Since $p \ne 0$,
$C$ contains at most 1 such codeword and it follows that $rs(B_c) \le 1$. On the
other hand, $y \ge \frac{2n-2m-4}{3} + p \ge p \ge 1$.  Since $C$ does not contain
an isolated codeword, $rs(B_c) = 1$. Thus $C$ contains one codeword that is neither
row-isolated or column-isolated, call it $v$, and we can write
$C = A_c \cup B_c \cup \{v\}$. Since $v$ is not column-isolated and $p=1$ then
$cs(A_c) = 2$. This implies that $|C| = m-1 + n-3 = m+n -4$. So we have
$m+n-4 \le \frac{2m+2n+1}{3} - 1$ which implies $m+n \le 10$, a contradiction.

\item[(iii)] Since $p=1$ and $q=1$, then $x \ge \frac{4m-2n-4}{3} +1$ and
$y \ge \frac{4n-2m-4}{3} + 1$. It follows that
$$\frac{2m+2n+1}{3}-1 \ge |C| \ge x+y \ge \frac{2m+2n+1}{3} - 1.$$
 Thus, $C = A_c \cup B_c$. On the other hand, since $p=1$ then $cs(A_c) = 1$.
 Since $cs(C) = n-1$, then $B_c$ must span the remaining $n-2$ columns. So
 $n-2 = \frac{4n-2m-4}{3} + 1$
which implies $m<n$, a contradiction.
\end{enumerate}
\end{enumerate}

Therefore, every ID code of $K_n \times K_m$ has cardinality at least
$\lceil \frac{2m+2n}{3}\rceil$. \\

We now present ID codes to show that this lower
bound is realized. \\
If $m \ne 2n-2$, let
$$D_1 =\{(1,1)\} \cup \{(i, 2i), (i, 2i+1)\,\big|\, 1 \le i \le a\} \cup
\{(a+2j-1, 2a+j+1), (a+2j, 2a+j+1)\big| 1 \le j \le b\}\,,$$
where $a=\frac{2m-n-2}{3}$ and $b = \frac{2n-m+1}{3}$.
It is straightforward to check that $D_1$ satisfies the properties of
Proposition~\ref{prop:fullspan} and is
therefore an ID code of $K_n \times K_m$.\\
If $m = 2n-2$, let
$$D_2 = \{(1,1)\} \cup \{(i, 2i), (i, 2i+1)\,\big| \,1 \le i \le n-2\} \cup \{(n-1,2n-2),(n,2n-2)\}\,.$$
Again, one can verify that $D_2$ satisfies all properties of
Proposition~\ref{prop:fullspan} and is therefore an ID code of $K_n\times K_{2n-2}$.

Therefore, if $m \ne 2n-5$ but $n + m \equiv 1 \pmod{3}$ and $6 \le n \le m \le 2n-2$,
then
\[\gid(K_n \times K_m) = \left\lceil\frac{2m+2n}{3}\right\rceil\,.\]
\end{proof}

Figure~\ref{fig:IDCodes2} contains examples of minimum cardinality ID codes for some cases
covered in Theorem~\ref{thm:1mod3}.  As in Figure~\ref{fig:IDCodes1} the code 
consists of the solid vertices.

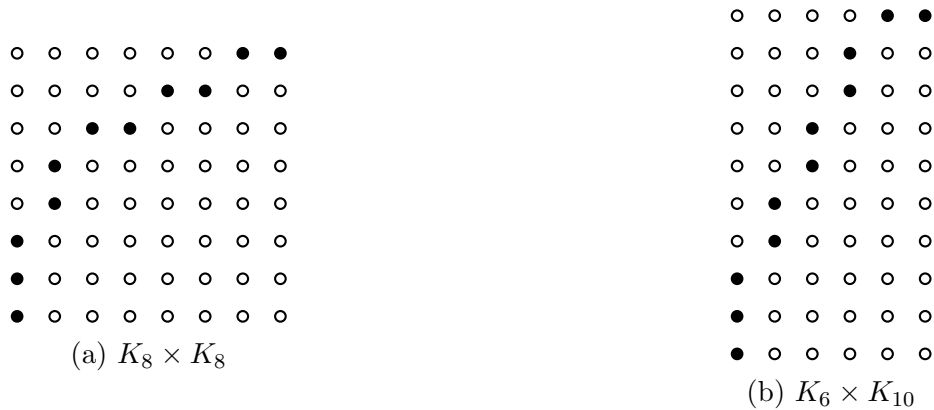
\begin{figure}[ht!]
\begin{subfigure}[]{.5\textwidth}
\centering
\begin{tikzpicture}[scale=1.0,style=thick]
\def\vr{2pt} 
\draw (0,0) [fill=black] circle (\vr); \draw (.5,0) [fill=white] circle (\vr);
\draw (1,0) [fill=white] circle (\vr); \draw (1.5,0) [fill=white] circle (\vr);
\draw (2,0) [fill=white] circle (\vr); \draw (2.5,0) [fill=white] circle (\vr);
\draw (3,0) [fill=white] circle (\vr); \draw (3.5,0) [fill=white] circle (\vr);

\draw (0,.5) [fill=black] circle (\vr); \draw (.5,.5) [fill=white] circle (\vr);
\draw (1,.5) [fill=white] circle (\vr); \draw (1.5,.5) [fill=white] circle (\vr);
\draw (2,.5) [fill=white] circle (\vr); \draw (2.5,.5) [fill=white] circle (\vr);
\draw (3,.5) [fill=white] circle (\vr); \draw (3.5,.5) [fill=white] circle (\vr);

\draw (0,1) [fill=black] circle (\vr); \draw (.5,1) [fill=white] circle (\vr);
\draw (1,1) [fill=white] circle (\vr); \draw (1.5,1) [fill=white] circle (\vr);
\draw (2,1) [fill=white] circle (\vr); \draw (2.5,1) [fill=white] circle (\vr);
\draw (3,1) [fill=white] circle (\vr); \draw (3.5,1) [fill=white] circle (\vr);

\draw (0,1.5) [fill=white] circle (\vr); \draw (.5,1.5) [fill=black] circle (\vr);
\draw (1,1.5) [fill=white] circle (\vr); \draw (1.5,1.5) [fill=white] circle (\vr);
\draw (2,1.5) [fill=white] circle (\vr); \draw (2.5,1.5) [fill=white] circle (\vr);
\draw (3,1.5) [fill=white] circle (\vr); \draw (3.5,1.5) [fill=white] circle (\vr);

\draw (0,2) [fill=white] circle (\vr); \draw (.5,2) [fill=black] circle (\vr);
\draw (1,2) [fill=white] circle (\vr); \draw (1.5,2) [fill=white] circle (\vr);
\draw (2,2) [fill=white] circle (\vr); \draw (2.5,2) [fill=white] circle (\vr);
\draw (3,2) [fill=white] circle (\vr); \draw (3.5,2) [fill=white] circle (\vr);

\draw (0,2.5) [fill=white] circle (\vr); \draw (.5,2.5) [fill=white] circle (\vr);
\draw (1,2.5) [fill=black] circle (\vr); \draw (1.5,2.5) [fill=black] circle (\vr);
\draw (2,2.5) [fill=white] circle (\vr); \draw (2.5,2.5) [fill=white] circle (\vr);
\draw (3,2.5) [fill=white] circle (\vr); \draw (3.5,2.5) [fill=white] circle (\vr);

\draw (0,3) [fill=white] circle (\vr); \draw (.5,3) [fill=white] circle (\vr);
\draw (1,3) [fill=white] circle (\vr); \draw (1.5,3) [fill=white] circle (\vr);
\draw (2,3) [fill=black] circle (\vr); \draw (2.5,3) [fill=black] circle (\vr);
\draw (3,3) [fill=white] circle (\vr); \draw (3.5,3) [fill=white] circle (\vr);

\draw (0,3.5) [fill=white] circle (\vr); \draw (.5,3.5) [fill=white] circle (\vr);
\draw (1,3.5) [fill=white] circle (\vr); \draw (1.5,3.5) [fill=white] circle (\vr);
\draw (2,3.5) [fill=white] circle (\vr); \draw (2.5,3.5) [fill=white] circle (\vr);
\draw (3,3.5) [fill=black] circle (\vr); \draw (3.5,3.5) [fill=black] circle (\vr);
\end{tikzpicture}
\caption{$K_8\times K_8$}
\end{subfigure}~
\begin{subfigure}[]{.5\textwidth}
\centering
\begin{tikzpicture}[scale=1.0,style=thick]
\def\vr{2pt} 
\draw (0,0) [fill=black] circle (\vr); \draw (.5,0) [fill=white] circle (\vr);
\draw (1,0) [fill=white] circle (\vr); \draw (1.5,0) [fill=white] circle (\vr);
\draw (2,0) [fill=white] circle (\vr); \draw (2.5,0) [fill=white] circle (\vr);

\draw (0,.5) [fill=black] circle (\vr); \draw (.5,.5) [fill=white] circle (\vr);
\draw (1,.5) [fill=white] circle (\vr); \draw (1.5,.5) [fill=white] circle (\vr);
\draw (2,.5) [fill=white] circle (\vr); \draw (2.5,.5) [fill=white] circle (\vr);

\draw (0,1) [fill=black] circle (\vr); \draw (.5,1) [fill=white] circle (\vr);
\draw (1,1) [fill=white] circle (\vr); \draw (1.5,1) [fill=white] circle (\vr);
\draw (2,1) [fill=white] circle (\vr); \draw (2.5,1) [fill=white] circle (\vr);

\draw (0,1.5) [fill=white] circle (\vr); \draw (.5,1.5) [fill=black] circle (\vr);
\draw (1,1.5) [fill=white] circle (\vr); \draw (1.5,1.5) [fill=white] circle (\vr);
\draw (2,1.5) [fill=white] circle (\vr); \draw (2.5,1.5) [fill=white] circle (\vr);

\draw (0,2) [fill=white] circle (\vr); \draw (.5,2) [fill=black] circle (\vr);
\draw (1,2) [fill=white] circle (\vr); \draw (1.5,2) [fill=white] circle (\vr);
\draw (2,2) [fill=white] circle (\vr); \draw (2.5,2) [fill=white] circle (\vr);

\draw (0,2.5) [fill=white] circle (\vr); \draw (.5,2.5) [fill=white] circle (\vr);
\draw (1,2.5) [fill=black] circle (\vr); \draw (1.5,2.5) [fill=white] circle (\vr);
\draw (2,2.5) [fill=white] circle (\vr); \draw (2.5,2.5) [fill=white] circle (\vr);

\draw (0,3) [fill=white] circle (\vr); \draw (.5,3) [fill=white] circle (\vr);
\draw (1,3) [fill=black] circle (\vr); \draw (1.5,3) [fill=white] circle (\vr);
\draw (2,3) [fill=white] circle (\vr); \draw (2.5,3) [fill=white] circle (\vr);

\draw (0,3.5) [fill=white] circle (\vr); \draw (.5,3.5) [fill=white] circle (\vr);
\draw (1,3.5) [fill=white] circle (\vr); \draw (1.5,3.5) [fill=black] circle (\vr);
\draw (2,3.5) [fill=white] circle (\vr); \draw (2.5,3.5) [fill=white] circle (\vr);

\draw (0,4) [fill=white] circle (\vr); \draw (.5,4) [fill=white] circle (\vr);
\draw (1,4) [fill=white] circle (\vr); \draw (1.5,4) [fill=black] circle (\vr);
\draw (2,4) [fill=white] circle (\vr); \draw (2.5,4) [fill=white] circle (\vr);

\draw (0,4.5) [fill=white] circle (\vr); \draw (.5,4.5) [fill=white] circle (\vr);
\draw (1,4.5) [fill=white] circle (\vr); \draw (1.5,4.5) [fill=white] circle (\vr);
\draw (2,4.5) [fill=black] circle (\vr); \draw (2.5,4.5) [fill=black] circle (\vr);
\end{tikzpicture}
\caption{$K_6\times K_{10}$}

\end{subfigure}
\caption{Several ID codes when $n+m \equiv 1\pmod{3}, m \ne 2n-5$}
\label{fig:IDCodes2}
\end{figure}

\bibliographystyle{plain}
\bibliography{mainidc}
\end{document}